\documentclass[12pt]{article}

\usepackage{amsmath,amsthm,amssymb,amsfonts}
\usepackage{todonotes, lipsum}
\usepackage[T1]{fontenc}\usepackage[upright]{fourier}\usepackage{baskervald}
\usepackage{dsfont}

\usepackage[all]{xy}

\usepackage{tikz-cd}
\usetikzlibrary{arrows}

\usepackage[normalem]{ulem}

\usepackage{verbatim}

\theoremstyle{plain}
\newtheorem{thm}{Theorem}[section]

\newtheorem{prop}[thm]{Proposition}

\theoremstyle{remark}

\newenvironment{frcseries}{\fontfamily{frc}\selectfont}{}

\usepackage{geometry}
\usepackage{tkz-graph}
\tikzstyle{vertex}=[circle, draw, inner sep=0pt, minimum size=6pt]
\begin{document}

\title{
Birationality and Landau-Ginzburg models
}
\author{Patrick Clarke}
\date{\today}

\maketitle

\begin{abstract}
We introduce a new technique for approaching birationality questions that arise in the mirror symmetry of complete intersections in toric varieties.  As an application we answer affirmatively and conclusively the question of Batyrev-Nill \cite{Batyrev-Nill} about the birationality of Calabi-Yau families associated to multiple mirror nef-partitions.  This completes the progress in this direction made by  Li's breakthrough  \cite{li-2016}.  In the process, we obtain results in the theory of Borisov's nef-partitions \cite{Borisov} and provide new insight into the geometric content of the multiple mirror phenomenon.
\end{abstract}

\section{Introduction.}

This paper provides a birationality result (Theorem \ref{theorem:omega-iso}) for a certain situation that arises in Batyrev-Nill's \cite{Batyrev-Nill} multiple mirror nef-partitions.  
We are able to successfully apply this theorem to prove the birationality of the multiple mirror  Calabi-Yau families (Theorem \ref{theorem:birat-nef}).
This completes the progress in this direction made by  Li's breakthrough  \cite{li-2016} by proving the result in full generality and without assumptions.  
Our methods 
 add a new perspective on how to connect data from one geometric phase to another in a toric Landau-Ginzburg model.
In addition, a  ``coarsening'' technique (Proposition \ref{prop:coarsening}) and directed graphs extracted from multiple mirror nef-partitions (Subsection \ref{subsection:d-graph}) contribute to the developing literature (e.g. Batyrev-Borisov \cite{Batyrev-Borisov-Cones},
Kreuzer-Reigler-Sahakyan \cite{kreuzer-reigler-sahakyan},
Nill-Schepers \cite{nill-schepers}) on nef-partitions themselves.

The title of this paper addresses the fact that the situation described by our theorem arises more generally when a variation of the K\"ahler parameter  induces a geometric transition from one Landau-Ginzburg model corresponding to a complete intersection to another.  
In order to not  go too far afield, we refer the interested reader to Clarke \cite{clarke-2008} for this more general setting
and Witten \cite{witten-phases} for an early physical treatment.  
In the last section of this paper, we make some remarks about questions raised by our theorem in the conjecturally vacuous theory of multiple mirror Fano varieties.
 
The multiple mirror nef-partitions of  Batyrev-Nill \cite{Batyrev-Nill} appeared as a phenomenon that arises in Borisov's  \cite{Borisov} formulation of mirror symmetry for Calabi-Yau complete intersections in toric varieties.  
 In such a situation, a single family  of Calabi-Yau complete intersections is mirror to two different families of Calabi-Yaus. 
 The work of  Batyrev-Borisov \cite{batyrev-borisov-on-cy} on the Hodge theoretical mirror symmetry of 
 Dixon \cite{dixon-1987}, Lerche-Vafa-Warner \cite{lerche-vafa-warner-1989} and Greene-Plesser \cite{Greene_Plesser_1990}
 guarantees that members of these two families have the same stringy Hodge numbers, and Batyrev's work on birational Calabi-Yaus \cite{batyrev-birational} suggests that these multiple mirror families might be birational.  
 This question of birationality and a related one about derived categories were asked by Batyrev-Nill \cite{Batyrev-Nill} in the same paper that first observed the phenomenon.

Similar examples of this kind of birationality appear in the theory of Berglund-H\"ubsch-Krawitz mirror symmetry \cite{Berglund-Hubsch, krawitz:2010}.  
Birationality in these cases was shown first by Shoemaker \cite{shoemaker-2014}, and then developed further by Kelly \cite{kelly-2013} and Clarke \cite{clarke-2013}.  
The general phenomenon seems to be that the superpotential of the Landau-Ginzburg model governs the birationality class of the underlying sigma model, and if multiple mirror examples are found in any of the toric mirror constructions (e.g. Greene-Plesser \cite{Greene_Plesser_1990},
Berglund-H\"ubsch \cite{Berglund-Hubsch},
Batyrev \cite{Batyrev},
Borisov \cite{Borisov},
Givental \cite{Givental-1998},
Hori-Vafa \cite{Hori-Vafa}, and
Clarke \cite{clarke-dual-fans})
our technique can be applied.

Related results have been found comparing the derived categories of these different geometries.  
This is suggested by Kontsevich's homological mirror symmetry conjecture  \cite{kontsevich-hms}, and became accessible to modern techniques with the developments of Hori-Herbst-Page \cite{hori-herbst-page} and Ballard-Favero-Katzarkov \cite{ballard-favero-katzarkov} on derived categories and variation of GIT quotients.  
Batyrev-Nill \cite{Batyrev-Nill} also suggested that the derived categories of these multiple mirrors might be equivalent.  This was answered in the affirmative by the papers of Kelly-Favero \cite{Favero-Kelly:BN, Favero-Kelly:BHK}.  In fact, Doran-Favero-Kelly \cite{doran-favero-kelly} seems to have completely settled these sort of questions about derived categories for multiple mirrors that arise from variation of the K\"ahler parameter.

Our methods are mostly elementary.  We use basic toric geometry in a style that is familiar in the theory of toric Landau-Ginzburg models.  The matrix $(W)$ introduced in Subsection \ref{subsection:expansions} is inspired by, but different than one used by Li \cite{li-2016}.  Our directed graph $D$ in Subsection \ref{subsection:d-graph} is new, but our treatment requires little knowledge graph theory; indeed, we appeal only to the Perron-Frobenius theorem.  
The one area where some familiarity beyond what is contained in this paper might be helpful is the theory of nef-partitions; we feel that it is hard to operate confidently with these objects without having spent some time manipulating and experimenting with them.

{\bf Acknowledgements.}  
We would like to thank Ron Donagi, Zhan Li, and Jonah Blasiak for their contribution to various parts of this work.  
First and foremost, conversations and correspondence with Ron Donagi  have been invaluable.
In addition, we would like to thank him for teaching us the lucid argument for the irreducibility of the universal complete intersection that appears in the proof of Theorem \ref{theorem:birat-nef}.  
Zhan Li has offered feedback at all stages of the project.
His input has motivated us to considerably improve and strengthen many parts of this paper.
Finally, it was at Jonah Blasiak's suggestion that we used the Perron-Frobenius theorem in the proof of Proposition \ref{proposition:omega-non-empty}.  Before this, the proof was considerably more ad hoc.

%%%%%%%%%%%%%%%%%%%%%%%%%%%%%%%%%%%%%%%%%%%%%%%%%%%%%%%%%%%%%%%%%%%%%%%%%%%%%%%%%%%%%%%%%%%%%%%%%%%%%%%%%%%%%%%%%%%%%%%%%%%%%%%%%%%%%%%%%%%%%%%%%%%%%%%%%%%%%%%%%%%%%%%%%%%%%%%%%%%%%%%%%%%%%%%%%%%%%%%%%%%%%%%%%%%%%%%%%%%%

\section{Toric varieties in brief.}

Here we summarize aspects of toric geometry that we will later take for granted.
These include the expansion of a  function in characters, the toric structure on completely split vector bundles, notational conventions for characters, and the roll of polytopes in the theory of projective toric varieties.  These facts are standard and can be found in Cox-Little-Schenck \cite{cox-little-schenck}, as well as Oda \cite{oda-toric} and Fulton \cite{fulton-toric}.

{\bf Toric varieties and character expansions.} 
A toric variety $Y$ is a normal algebraic variety containing an algebraic torus $T \cong \mathbf{G}_m^{\dim Y}$ as an open dense subset, such that the action of the torus on itself extends to the whole variety.  
The regular functions on $T$ have characters $\chi \colon T \to \mathbf{G}_m$ as a basis.  
Thus any rational  function $f \colon Y \dashrightarrow \mathbf{A}^1$ can be expanded uniquely in characters.

{\bf Line bundles and total spaces.}   
The total space of line bundle $\mathcal{L}$ over $Y$ can always be given the structure of a toric variety.  
Recall that the total space $X$ of $\mathcal{L}$ is the relative spectrum of the sheaf of algebras $\operatorname{Sym}^\bullet \mathcal{L}^\vee$, where $\mathcal{L}^\vee$ is the dual bundle $\text{\it Hom}_{Y}(\mathcal{L}, \mathcal{O}_Y).$ 
The toric structure comes from choosing a non-zero rational section $e$ of $\mathcal{L}$ whose divisor is invariant under the action of $T$.  
Then the torus $\mathbf{G}_m \times T$ includes into $X$ by sending 
$(\lambda, t)$ to  $\lambda \cdot e(t)$. 
Iterating this procedure, the total space of any completely split bundle over a toric variety is toric.

{\bf Character notations.} 
Traditionally, the set of characters  on the torus $T$ is denoted $M$ and it is a finite rank abelian group whose operation is denoted $+$.  
When considering $m \in M$ as a rational function on $Y$ we use the notation $\chi^m$.  
This way sums in $M$ become products of functions in $\mathbf{C}(Y)$,
$$
\chi^{m+m'} = \chi^m \chi^{m'},
$$
 and there is no confusion between the sum of functions $\chi^m + \chi^{m'}$ and the ``sum'' of characters $\chi^{m+m'}.$

{\bf Polytopes and nef divisors.} 
On a projective toric variety a Cartier divisor is nef if and only if it is base point free.  In addition, any divisor is linearly equivalent to a torus invariant divisor.

The projective toric varieties which compactify a torus $T$ and the  monoid of torus invariant nef divisors on them can be completely encoded into convex, compact, integral polytopes in $M_\mathbf{R}$.
A torus invariant nef divisor $D$ on a projective toric variety $Y$ determines and is determined by a convex, compact, integral, polytope $P_D \subseteq M_\mathbf{R}$. 
The integral points of $P_D$ are those characters $m$ for which $(\chi^m) + D$ is effective, and therefore give a basis for the global sections of $\mathcal{O}(D)$.
This construction behaves well with sums in that $P_{D+D'}$ equals the Minkowski sum $P_D + P_{D'}$.
Thus the ring of global functions on the  total space of $\mathcal{O}(-D)$ is 
$\mathbf{C}[\sigma_{P_D} \cap (M \oplus \mathbf{Z})]$
where 
$$\sigma_{P_D} = \{  (rP_D, r) \in  M_\mathbf{R} \oplus \mathbf{R} \ | \ 0 \leq r \in \mathbf{R} \}.$$
Finally, note that torus invariant nef divisors $D$ and $D'$ are linearly equivalent if and only if $P_{D'} = P_D + m$ for some $m \in M$.

Reversing this point of view, a full dimensional, convex, compact, integral, polytope $P \subseteq M_{\mathbf{R}}$ determines a toric variety as 
$Y_P = \operatorname{Proj}  \mathbf{C}[\sigma_{P} \cap (M \oplus \mathbf{Z})]$.  
Finally, a polytope is the polytope of a nef divisor on $Y_P$ if and only if it can be written as the intersection of integral translates of the supporting half-spaces containing the facets of $P$.

\section{Geometries, expansions, assumptions and geography.}
\label{section:geometries-expansions-and-assumptions}

This section sets the stage for our general birationality result in  Section \ref{section:birationality}.
This is a  statement about complete intersections in toric varieties.  
So for $i=1,2$,
the input data includes
\begin{itemize}
\item a base toric variety $Y^{(i)}$,
\item a completely split bundle $\mathcal{V}^{(i)}$,
and
\item a global section $g^{(i)}$ in $\Gamma(Y^{(i)}, \mathcal{V}^{(i)})$.
\end{itemize}
The complete intersection $Z^{(i)} \subseteq Y^{(i)}$ is then the vanishing of the section $g^{(i)}$.

This data is not arbitrary.  In fact we impose the strong constraint that 
the sections $g^{(1)}$ and $g^{(2)}$ are ``the same.'' 
This is spelled out precisely in 
Subsection \ref{subsection:geometries}.
In short,  $g^{(i)}$ defines a function on the total space $X^{(i)}$ of the dual bundle ${\mathcal{V}^{(i)}}^\vee$. In our situation,
 $X^{(1)}$ and $X^{(2)}$ are birational, and the functions defined by $g^{(1)}$ and $g^{(2)}$ are the same 
on the open set identified via the birational map.

The common function defined by the bundle sections is written $W$.  
The birational map between $X^{(1)}$ and $X^{(2)}$ identifies their tori.  
Once identified, we refer to this simply as $T$. 
On $T$, the function $W$ can be expanded in characters.  
In addition, each complete intersection assigns a different grouping of the terms of this expansion.  
Subsection \ref{subsection:expansions} provides the details, introduces a matrix $(W)$ and a directed graph $D$ that organize the term groupings.  
Ultimately, an analysis of the graph $D$  allows us to apply our result to multiple mirror nef-partitions.  
Subsection \ref{subsection:assumptions} details two assumptions about the characters in the expansions.  
These are mild, and are always satisfied for multiple mirror nef-partitions after specializing to members in the mirror families.

Identification of the total spaces along the torus $T$ allows us to talk about points $t \in T$ that lie over one or both complete intersections.  
Within $T$ and above the complete intersections is a special subset we denote $R^{(1)} \cap R^{(2)}$. 
The amazing fact behind the proof of the birationality result is that there is a torus  $\mathbf{G}_m^\beta$ which acts on $R^{(1)} \cap R^{(2)}$
such that 
$$
R^{(1)} \cap R^{(2)} \ \cong \ \Omega^{(1)} \times  \mathbf{G}_m^\beta \ \cong \ \Omega^{(2)} \times \mathbf{G}_m^\beta
$$
as 
$\mathbf{G}_m^\beta$-spaces, for open subsets $\Omega^{(i)} \subseteq Z^{(i)}$.
Subsection \ref{subsection:geography} defines the $R$ and $\Omega$-sets, as well as a few other sets that help us in our arguments.

\subsection{Geometries.}
\label{subsection:geometries}
This subsection states precisely the geometry of the situation in which we can apply our theorem, and describes the identification of the total spaces along an open dense torus $T$.

\label{subsection:geometries}
{\bf Initial data.}
We begin with two toric varieties $Y^{(i)}$ for $i=1,2$.  
Over each of these we have  completely split rank $r$ bundles  
$$
\mathcal{V}^{(i)} = \mathcal{L}^{(i)}_1 \oplus \dotsm \oplus \mathcal{L}^{(i)}_r
$$ 
The line bundles $\mathcal{L}$ are assumed to be nef.
In addition, each bundle is equipped with a section $g^{(i)} \in \Gamma(Y^{(i)},\mathcal{V}^{(i)})$.
The complete intersections $Z^{(i)}$ that will be compared are given as the vanishing loci $Z(g^{(i)})$ of the sections $g^{(i)}.$

{\bf Toric total spaces and superpotential morphisms.}
From the initial data we form the total space of ${\mathcal{V}^{(i)}}^\vee$:
$$
X^{(i)} = \operatorname{Spec}_{Y^{(i)}} \operatorname{Sym}^\bullet \mathcal{V}^{(i)}.
$$
To this we add the morphism 
$$W^{(i)} \colon X^{(i)} \to \mathbf{A}^1$$
 defined by considering 
$g^{(i)}$ as a function.

{\bf Equal affinization and superpotential.}
The connection between these two sets of data is twofold.  First, we require  the affinizations of $X^{(1)}$ and $X^{(2)}$ to be isomorphic as toric varieties.  
This means that the monoid of global characters in 
$$
\Gamma(Y^{(1)}, \operatorname{Sym}^\bullet \mathcal{V}^{(1)})
$$
is isomorphic to the monoid of global characters in 
$$
\Gamma(Y^{(2)}, \operatorname{Sym}^\bullet \mathcal{V}^{(2)}).
$$
Geometrically, this means there is a single affine toric variety $X^{(0)}$ and toric isomorphisms
$$
X^{(0)} \cong \operatorname{Spec} \Gamma(Y^{(i)}, \operatorname{Sym}^\bullet \mathcal{V}^{(i)}).
$$ 
Consequentially, the tori $T^{(1)}$, $T^{(2)}$ and $T^{(0)}$ are isomorphic.  Note that it is the fact that the line bundles are nef, and therefore generated by global sections, that guarantees that tori $T^{(1)}$ and $T^{(2)}$ map isomorphically to $T^{(0)}$ under the universal map to the affinization.

The second connection between the two sets of data is the requirement that the morphisms $W^{(i)}$ are  related through $X^{(0)}$.  The fact that $X^{(0)}$ is the affinization of both $X^{(i)}$'s guarantees that both the $W^{(i)}$'s are global functions in $\mathcal{O}_{X_0}$. Therefore, the both define morphisms $X_0 \to \mathbf{A}^1$.  We require that these two morphism agree.  From now on, we will simply write $W$ for $W^{(i)}$.

Writing simply $T$ for the isomorphic tori $T^{(0)}, T^{(1)}$ and $T^{(2)}$, these spaces fit together in the commutative diagram:
\begin{equation}
\label{equation:X's-diagram}
\begin{tikzcd}
\  &  &T  \arrow[left hook->]{dl}  \arrow[hook]{dr}  \\
 &     X^{(1)} \arrow[two heads]{dr} \arrow[two heads]{d} & & X^{(2)} \arrow[two heads]{dl}  \arrow[two heads]{d} \\
Z^{(1)} \arrow[hook]{r} &     Y^{(1)}  & X^{(0)} \arrow{d}{W}& Y^{(2)} &\arrow[left hook->]{l} Z^{(2)}\\
  &     & \ \mathbf{A}^1 .\\
\end{tikzcd}
\end{equation}
In addition, each $Y^{(i)}$ has its own torus $T_{Y^{(i)}}$, and the maps $T \to T_{Y^{(i)}}$  are surjective homomorphisms.

{\bf Batyrev-Nill's Example 5.1: $W \colon X^{(0)} \to \mathbf{A}^1$.}  The original example of Batyrev and Nill proposes that the family of complete intersections given by the bundle
$$
\mathcal{V}^{(1)} = \mathcal{O}(1) \oplus \mathcal{O}(1) \text{ over } \mathbf{P}^3 
$$
restricted to the surface 
$$
Y^{(1)} = \{ [X:Y:Z:W] \ | \ Y^2 = XZ \}
$$
has the same mirror family as the family of complete intersections given by the bundle 
$$
\mathcal{V}^{(2)} = \mathcal{O}(2,1) \oplus \mathcal{O}(0,1) \text{ over } Y^{(2)} = \mathbf{P}^1 \times \mathbf{P}^1.
$$
In this example, the affine toric variety $X^{(0)}$ is the spectrum of the subring
$$
\Gamma(Y^{(1)}, \operatorname{Sym}^\bullet \mathcal{V}^{(1)}) \cong 
\mathbf{C}[
\underbrace{\frac{y}{x} s, y s, xy s, s}_{\Gamma(Y^{(1)}, \mathcal{O}(1))}
,  
\underbrace{\frac{t}{x},  x t, \frac{t}{y}, t}_{\Gamma(Y^{(1)}, \mathcal{O}(1))}
] 
\subseteq \mathbf{C}(x,y,s,t)
$$
which is isomorphic to the subring
$$
\Gamma(Y^{(2)}, \operatorname{Sym}^\bullet \mathcal{V}^{(2)}) \cong 
\mathbf{C}[
\underbrace{\frac{w}{z}t, wt, zwt, \frac{t}{z}, zt, t}_{\Gamma(Y^{(2)}, \mathcal{O}(2,1))},
\underbrace{\frac{s}{w}, s}_{\Gamma(Y^{(2)}, \mathcal{O}(0,1))}
] 
\subseteq \mathbf{C}(z,w,s,t)
$$
via the substitutions
$$
x = z \text{ \ \ and \ \ } y = w \frac{t}{s} \ \ .
$$
In both cases, the sections $s$ and $t$ trivalize the line bundles over the torus.  The functions $x$ and $y$ give character coordinates on $Y^{(1)}$, and the
functions $z$ and $w$ give character coordinates on $Y^{(2)}$.

The superpotential with generic coefficients $c$ is the sum of eight terms
$$
\begin{array}{rcllllllll}
W &  = &  
c_{(-1, 1)} \frac{y}{x} s & + c_{(0,1)} y s &  +  c_{(1,1)} xy s &  + 
c_{(-1,0)} \frac{t}{x} &  +   c_{(1,0)} x t & + c_{(0,-1)} \frac{t}{y} &  + 
c_s s &  +  c_t t 
\\
& = & 
c_{(-1,1)} \frac{w}{z}t &  + c_{(0,1)} wt &  + c_{(1,1)} zwt & + c_{(-1,0)} \frac{t}{z} & + c_{(1,0)} zt & + 
c_{(0,-1)} \frac{s}{w}  & + c_s s  & + c_t t .
\end{array}
$$
Later, we will see that the indexing of the coefficients comes from the lattice points in the convex hull of the polytopes making up the relevant nef-partitions.  The coefficients $c_s$ and $c_t$ are indexed differently because they will correspond to the origin in the nef-partition, and there is one origin term for each polytope.
These origin terms always correspond to the local trivializing sections of the associated line bundles.

\subsection{Expansions.}
\label{subsection:expansions}

There are several expansions of the function $W$ that are important.  
These come initially from three sources: the toric structure of the $X$'s, and the splittings of $\mathcal{V}^{(i)}$'s into the $\mathcal{L}^{(i)}$'s for $i = 1,2$.  
Finally, $W$ is expanded in a way that respects all three interpretations simultaneously.

An important point is that these expansions correspond to  set partitions of the characters $\Xi$
appearing with non-zero coefficients in the expansion of $W$.  In other words, a character appears in at most one term of any of our expansions.

It is  convenient to work on the torus $T$.  
%Note however, the characters and thus all the terms in all the expansions of $W$ extend to regular functions on $X^{(0)}$ and thus also the  $X^{(i)}$'s.
$T$  is an $n+k$ dimensional algebraic torus
and on it
any function on $X^{(0)}, X^{(1)}$  and $X^{(2)},$ can be considered as a function.  
Furthermore, $T$ maps surjectively onto the big tori $T_{Y^{(i)}}$, so functions on $Y^{(i)}$ can be considered on $T$ as well.

{\bf Characters.} First $W$ can be expanded in characters on $T$:
$$
W = \sum_{\chi \in \Xi} c_\chi \chi,
$$
where $\Xi$ is simply the set of characters appearing with non-zero coefficient.

{\bf Bundle splitting.} In addition, thinking of $W = g^{(i)} \in \mathcal{V}^{(i)} = \mathcal{L}^{(i)}_1 \oplus \dotsm \oplus \mathcal{L}^{(i)}_r$  we can write 
$$
g^{(i)} = \ell^{(i)}_1 + \dotsc + \ell^{(i)}_r
$$
for global sections $\ell^{(i)}_k \in \mathcal{L}^{(i)}_k$.  
Again each $\ell^{(i)}_k$ expands in characters, so we can view this part  of  $W$'s expansion.  This way for each $i = 1, 2$ we get two expansions
$$
W = \sum_k W^{(i)}_k
$$
where $W^{(i)}_k$ is the section $\ell^{(i)}_k$ considered as a function on $T$.

These bundle splitting expansions correspond to a grouping of the characters expanding $W$.
This is due to the uniqueness of the expansion of $W$ in characters and the fact that the line bundle sections cannot have characters in common.  
Indeed, each line bundle has its own ``locally trivializing'' character, that separates the terms of its section from the other line bundles.  
In symbols, as before we have
$$
\Xi = \{ \chi  |  \chi \text{ is a character with non-zero coefficient in the expansion of $W$ in characters} \},
$$
and now we have 
$$
\Lambda^{(i)}_k= \{ \chi \in \Xi | c_\chi \chi \text{ is a term in the character expansion of } W^{(i)}_k \}.
$$
Furthermore,
$$
\Xi = \coprod_{k =1}^r \Lambda^{(i)}_k.
$$

{\bf Simultaneous expansion.} 
Now we can refine the grouping of the terms of $W$.  For this we index the line bundles making up $\mathcal{V}^{(1)}$ by $a$, and those making up $\mathcal{V}^{(2)}$ by $b$.  
The simultaneous refinement of the two set partitions 
$$
\Xi = \coprod_{a} \Lambda^{(1)}_a = \coprod_{b} \Lambda^{(2)}_b
$$
leads to the expansion
$$
W = \sum_{a,b} W_{ab}
$$
where 
$$
W_{ab} = \sum_{\chi \in \Lambda^{(1)}_a \cap \Lambda^{(2)}_b} c_\chi \chi. 
$$
Finally, grouping these terms gives back the terms bundle splitting expansions:
$$
W^{(1)}_a = \sum_b W_{ab}
$$
and 
$$
W^{(2)}_b = \sum_a W_{ab}.
$$

\subsubsection{Organizing the $W_{ab}$'s: $D$ and $(W)$.}
Juggling the behavior of these expansions is central to our treatment.  We introduce a directed graph $D$ that organizes the terms.  The $W_{ab}$'s are then arranged in a matrix $(W)$, and the graph $D$ governs its structure.  Ultimately, it is the insight we get into the terms from analyzing the graph that allows us to prove our results and apply them to nef-partitions.

{\bf Graph and matrix.} The directed graph $D$ has vertices indexed $1, \dotsc, r$  and an arrow from $a$ to $b$ when $\Lambda^{(1)}_a \cap \Lambda^{(2)}_b \neq \varnothing.$  Equivalently, there is an arrow from $a$ to $b$ when $W_{ab} \not \equiv 0$.  
$D$ splits into connected components $D_{(j)}$, and appropriately indexing the $\mathcal{L}$'s  
we can make the matrix 
 $$
(W) = (W_{ab})_{ab}
 $$
 into a block diagonal matrix whose blocks $(W_{(j)})$ correspond to the connected components $D_{(j)}$.  The row/column index set for the block is the vertex set of $D_{(j)}$, and the non-zero entries of the block correspond to $D_{(j)}$'s arrows.
 
 A crucial property of $(W)$ is that summing the entries in the $k^{\text{th}}$-row gives $W^{(1)}_k$, and
 summing the entries in the $k^{\text{th}}$-column gives $W^{(2)}_k$.

 We write $d_j$ for the number of vertices in $D_{(j)}$.  This means the index set for the $j^\text{th}$ block is
$$d_1+ \dotsc + d_{j-1}+1 \  , \  d_1+ \dotsc + d_{j-1}+2 \ , \ \dotsc \ , \ d_1+ \dotsc + d_{j-1}+d_j .$$
Finally, we write $\beta$ for the number of connected components in $D$, so the index $j$ runs from $1$ to $\beta$.

Although we will not use it in our birationality result and the application to nef-partitions, the expansion
$$
W = \sum_j W_{(j)}.
$$
where 
$$
W_{(j)} = \sum_{a,b \in D_{(j)}} W_{ab}
$$
is a ``block function,'' and the terms $W_{(j)}$ naturally arise in the discussion on multiple mirror Fano manifolds at the end of the paper.

{\bf Batyrev-Nill's Example 5.1: $(W)$ and $D$.} Recall that in the example of Bayrev and Nill we have the identifications $x = z$ and $y = w \frac{t}{s}$.  In this example we have
$$
(W) = 
\begin{pmatrix}
c_s s &  c_{(-1, 1)} \frac{y}{x} s  + c_{(0,1)} y s + c_{(1,1)} xys\\
 c_{(0,-1)} \frac{t}{y}  & c_{(-1,0)} \frac{t}{x}   +   c_{(1,0)} x t + c_t t 
\end{pmatrix}
=
\begin{pmatrix}
c_s s &  c_{(-1, 1)} \frac{w}{z} t  + c_{(0,1)} w t + c_{(1,1)} zwt \\
 c_{(0,-1)} \frac{s}{w}  & c_{(-1,0)} \frac{t}{z}   +   c_{(1,0)} z t + c_t t 
\end{pmatrix}
$$
and the graph $D$  is
$$
\begin{tikzpicture}[>=latex']

\node[vertex] (a) at (-1,0) {};
\node[vertex] (b) at (1,0) {};

\draw[->] (a) to [bend right=30] (b);
\draw[->] (b) to [bend right=30] (a);

\draw[->] (a) to [out=130, in= 220, looseness=15] (a);
\draw[->] (b) to [out=-50, in= 40, looseness=15] (b);
\end{tikzpicture}.
$$

\subsection{Assumptions and factorizations.}
\label{subsection:assumptions}
There are two assumptions required for the general theorem.  
The first assumption is always true for multiple mirror nef-partitions.
On the other hand the second assumption  does not hold
for multiple mirror nef-partitions until certain general coefficients are specialized, 
but this is easily dealt with in our application.

{\bf 1.} 
The first assumption is that for all $k$,
$$
\Lambda^{(1)}_k \cap \Lambda^{(2)}_k \neq \varnothing.
$$
Given this, we  make a choice  $\chi_{kk} \in \Lambda^{(1)}_k \cap \Lambda^{(2)}_k$.  Ultimately, our constructions are independent of this choice.  Note that if this condition isn't satisfied by the original indexing, it is possible that it might be after permuting the labels of the $\mathcal{L}^{(2)}$'s.  So we really require that this is true for some permutation of the $\mathcal{L}^{(2)}$ labels; in which case we relabel and then form the blocks $D_{(j)}$ etc.

{\bf 2.} The second assumption is that for $i=1,2$ the set
$$
\bigcup_k \{ \frac{\chi}{\ \chi'} \ | \ \chi, \chi' \in \Lambda^{(i)}_k \} 
$$
generates (by products) all characters on $T_{Y^{(i)}}$.

{\bf First and second factorization.}
Assumption 1 guarantees the matrix $(W)$ also admits two factorizations over the torus $T$.  The first is of the form 
$$
(W)  = \operatorname{diag}(x)F^{(1)}
$$
where the entries of the diagonal matrix $\operatorname{diag}(x)$ are the chosen characters $\chi_{kk}$, 
and the matrix $F^{(1)}$ depends only on the coordinates on $Y^{(1)}$.
The second is similar, but 
$$
(W)  = F^{(2)} \operatorname{diag}(x)
$$
where the diagonal matrix $\operatorname{diag}(x)$ is the same, and the matrix $F^{(2)}$
depends only on the coordinates on $Y^{(2)}$.  
These exist because they reflect the fact that when $\mathcal{L}^{(i)}_k$ is trivialized by the section corresponding to  $\chi_{kk}$
we can factor out $\chi_{kk}$ from the $k^{\text{th}}$ row in the $i=1$ case, and the $k^{\text{th}}$ column in the $i=2$ case.

{\bf Rescaling fibers.}
There are actions on the torus $T$ that come from ``rescaling'' in the fiber direction.  
We will use these in some of our arguments, so it is important that they are clearly understood.
The main thing to keep in mind is that there is not one fiber direction but two:
it depends on whether we are looking from the point of view of $X^{(1)}$ over $Y^{(1)}$ or $X^{(2)}$ over $Y^{(2)}$.
From the point of view of the first factorization, rescaling in the $X^{(1)}$ fiber direction amounts to rescaling the entries of $\operatorname{diag}(x)$ and leaving $F^{(1)}$ fixed.  
On the other hand, for the second factorization rescaling in the $X^{(1)}$ fiber direction will again rescale  the entries of $\operatorname{diag}(x)$, but the entries of $F^{(2)}$ will not be fixed.  Indeed, we have
$$
F^{(2)} = \operatorname{diag}(x) \ F^{(1)} \operatorname{diag}(x)^{-1},
$$
so $F^{(2)}$ is fixed by an $X^{(1)}$-rescaling if and only if the rescaling commutes with $F^{(1)}$---in particular if the scaling is constant on the blocks $(W_{(j)})$ of $(W)$.
Likewise for $X^{(2)}$-rescalings.

Finally note that the content of Assumption 2 is that all points in $Y^{(i)}$ are fixed by a rescaling if and only if $F^{(i)}$ is fixed by the same rescaling.  Indeed, if not there would be some character on $Y^{(i)}$ that would have non-zero weight under the rescaling.  This would be impossible if all the entries of $F^{(i)}$ had weight zero and their ratios generate all characters on $Y^{(i)}$, as Assumption 2 implies.

{\bf Batyrev-Nill's Example 5.1: factorizations.}  In the example of Batyrev and Nill, 
$$
\operatorname{diag}(x) = 
\begin{pmatrix}
s & 0 \\
0 & t
\end{pmatrix},
$$
and the factorizations are 
$$
\operatorname{diag}(x) F^{(1)} = 
\begin{pmatrix}
s & 0 \\
0 & t
\end{pmatrix}
\begin{pmatrix}
c_s  &  c_{(-1, 1)} \frac{y}{x}   + c_{(0,1)} y  + c_{(1,1)} xy \\
 c_{(0,-1)} \frac{1}{y}  & c_{(-1,0)} \frac{1}{x}   +   c_{(1,0)} x  + c_t  
\end{pmatrix} 
$$
and 
$$
F^{(2)} 
\operatorname{diag}(x) 
= 
\begin{pmatrix}
c_s  &  c_{(-1, 1)} \frac{w}{z}   + c_{(0,1)} w  + c_{(1,1)} zw\\
 c_{(0,-1)} \frac{1}{w}  & c_{(-1,0)} \frac{1}{z}   +   c_{(1,0)} z  + c_t  
\end{pmatrix}
\begin{pmatrix}
s & 0 \\
0 & t
\end{pmatrix} .
$$

\subsection{Geography.}
\label{subsection:geography}

There are several subsets of the torus $T$ that are important to our considerations.  
These all have descriptions in terms of the matrix $(W)$, however most  are naturally related to the complete intersection $Z$-sets.
Note that the openness of the $\Omega$ and $O$-sets below does not depend on our assumptions. 
Thus these sets will be open in the nef-partition application even before specializing coefficients.

{\bf Sets over $Z^{(i)}$.} The first subsets are those over the complete intersections $Z^{(i)}$.
Write $S^{(1)}$ for the subset of $T$ which lies over $Z^{(1)}$. In terms of the matrix $(W)$, we can identify $S^{(1)}$ with the set on which 
$$
(W) \begin{pmatrix} 1 \\ \vdots \\ 1 \end{pmatrix} = 0.
$$
In the same way we have the subset $S^{(2)}$ of $T$ which lies over $Z^{(2)}$, and similarly $S^{(2)}$ is the same as the set where 
$$
\begin{pmatrix} 1 , \dotsc , 1 \end{pmatrix} (W) = 0.
$$

{\bf Maximum rank subsets.}  Now that the $S$-sets are defined, the natural set to consider is $S^{(1)} \cap S^{(2)}.$
However, we will use the matrix $(W)$ to analyze this set, so we must first restrict to a subset on which $(W)$ is well behaved.
Specifically, on $S^{(i)}$ the rank of $(W)$ is at most $k-\beta$ where $k$ is the rank of $\mathcal{V}^{(i)}$ and $\beta$ is the number of blocks in the block diagonal decomposition of $(W)$.  

With this in mind, we define $R^{(i)}$ to to be the subset of $S^{(i)}$ on which $(W)$ has rank  $k-\beta$.  The set $R^{(i)}$ is open in $S^{(i)}$, and the subset of $S^{(1)} \cap S^{(2)}$ on which $(W)$ has rank $k-\beta$ is  $R^{(1)} \cap R^{(2)}.$  
Notice that $R^{(i)}$ is the preimage of an open set $\rho^{(i)}$ in $Z^{(i)}$ over which the matrix $F^{(i)}$ from the factorization of $(W)$ has rank $k-\beta$.  The openness of these sets can be seen by the fact that on or over $Z^{(i)}$, the rank $k-\beta$ subset is the same as the rank $\geq k-\beta$.

{\bf Unit entry null vector sets.}  
Now we focus attention on the intersection $R^{(1)} \cap R^{(2)}.$  
This is the subset of $S^{(1)} \cap S^{(2)}$ on which $(W)$ is well behaved;  that is to say these are the points in $T$ which lie over both $Z^{(1)}$ and $Z^{(2)}$ and where $(W)$ has rank $k - \beta.$

The image of $R^{(1)} \cap R^{(2)}$ in $Z^{(1)}$ is an open set $\Omega^{(1)} \subseteq \rho^{(1)}$.  In light of Assumption 1, this set can be characterized by the property that the $F^{(1)}$
matrix has rank $k - \beta$  and annihilates a row vector $h$:
$$
hF^{(1)} = 0
$$
whose entries are all non-zero.  In this case, setting $\chi_{kk} = h_k$ guarantees that $(W)$ annihilates both $(1, \dotsc, 1)$
and $\begin{pmatrix} 1 \\ \vdots \\ 1 \end{pmatrix}$.  Similarly, we have $\Omega^{(2)} \subseteq Z^{(2)}$ defined by 
 the property that the $F^{(2)}$
matrix has rank $k - \beta$  and annihilates a column vector $v$:
$$
F^{(2)} v = 0
$$
whose entries are all non-zero.

Finally, we denote the preimages in $T$ of the $\Omega$-sets by $O^{(i)}$, and so $R^{(1)} \cap R^{(2)} = O^{(1)} \cap O^{(2)}$.  
These $O$-sets are useful in that the non-emptyness of $\Omega^{(i)}$ is equivalent to showing the non-emptiness of the $O$'s.
This will be used in our application to multiple mirror nef-partitions.

{\bf Batyrev-Nill's Example 5.1: $Z^{(1)}, S^{(1)}$, $R^{(1)}$, $\Omega^{(1)}$, and $O^{(1)}$.}  In the example of Batyrev and Nill, the  set 
\begin{itemize}
\item 
$Z^{(1)}$ is given by 
$$
  F^{(1)}
  \begin{pmatrix}
  1 \\
  1
  \end{pmatrix}
   = 
\begin{pmatrix}
c_s  +  c_{(-1, 1)} \frac{y}{x}   + c_{(0,1)} y  + c_{(1,1)} xy  \\
 c_{(0,-1)} \frac{1}{y}  + c_{(-1,0)} \frac{1}{x}   +   c_{(1,0)} x  + c_t 
\end{pmatrix} 
=
  \begin{pmatrix}
  0 \\
  0
  \end{pmatrix},
  $$
  with neither $s$ nor $t$ involved,
\item $S^{(1)}$ satisfies the same equations as $Z^{(1)}$ but with $s$ and $t$ free,
  \item
  $R^{(1)}$ lies in $S^{(1)}$  and satisfies $F^{(1)} \neq 0$, 
\item
   $\Omega^{(1)}$ is the union of the open sets 
\begin{itemize}
\item[$\ast$]  $f^{(1)}_{11} f^{(1)}_{21} = c_s( c_{(-1,0)} \frac{1}{x}) \neq 0$, and  
\item[$\ast$] $f^{(1)}_{12} f^{(1)}_{22}  = (c_{(-1, 1)} \frac{y}{x}   + c_{(0,1)} y  + c_{(1,1)} xy)( c_{(-1,0)} \frac{1}{x}   +   c_{(1,0)} x  + c_t) \neq 0$
\end{itemize}
on $Z^{(1)}$ with $F^{(i)} = (f^{(i)}_{ab})_{ab}$,  and
\item $O^{(1)}$ is the same as $\Omega^{(1)}$ but with $s$ and $t$ free.
\end{itemize}

\section{Birationality.}
\label{section:birationality}
Since $R^{(1)} \cap R^{(2)}$ surjects onto $\Omega^{(1)}$ and $\Omega^{(2)}$ by definition, 
the last step is to  identify $R^{(1)}\cap R^{(2)}$ with a trivial $\mathbf{G}_m^{\beta}$-torsor.  
We will find that the sets $\Omega^{(1)}$ and  $\Omega^{(2)}$ are both equal to quotient $(R^{(1)}\cap R^{(2)}) / \mathbf{G}^{\beta}_m$, and thus isomorphic.

\begin{prop}{\bf ($O$ torsor)}
The torus $\mathbf{G}_m^{\beta}$ acts on  $R^{(1)} \cap R^{(2)}$ in such a way that 
 $R^{(1)} \cap R^{(2)}$ is identified with 
$$
\Omega^{(1)} 
\times
\mathbf{G}_m^{\beta} 
$$
on one hand, and
$$ 
\Omega^{(2)}
\times 
\mathbf{G}_m^{\beta} 
$$
on the other.
\end{prop}
\begin{proof}
An element $(\lambda_j)_j \in \mathbf{G}_m^{\beta}$ acts on a point $t = (y_1, x) = (y_2, x)$ by rescaling the $x$-variables corresponding to the block $(W_{(j)})$ by $\lambda_j$.  
To be sure that if this is done with $t= (y_1, x)$ it doesn't affect the $y_2$-variables, observe that any $y_2$-variable is a ratio of $x$-variables times a $y_1$-variable---this is the content of Assumption 2.  So this gives a well-defined action, and the  projections to $Y^{(1)}$ and $Y^{(2)}$ are invariant.

This has the effect of rescaling the block itself by $\lambda_j$, and doesn't change the rank or eigenspaces.  This means that $R^{(1)} \cap R^{(2)}$ preserved.
The identification of $R^{(1)} \cap R^{(2)}$ with the sets 
$$
\Omega^{(1)} \times
\mathbf{G}_m^{\beta} 
$$
requires the choice of a section over $\Omega^{(1)}$ in $R^{(1)} \cap R^{(2)}$.  
This means a choice of $x$-variables for each $z_1 \in \Omega^{(1)}$.  One of many ways to do this is to require that $x_k = 1$
whenever $k = d_1 + \dotsm + d_j$.  This means the last $x$-entry appearing in the $(W_{(j)})$ block is set to $1$.  The rank condition defining the $R$-sets guarantees that this fixes all the rest of the $x$-variables. This section is regular over $\Omega^{(1)}$ since the values of the other $x$-variables can be read off from an echelon form of $(W)$.  

The proof is complete by the identical argument for the $Z^{(2)}$ side.
\end{proof}

\begin{thm}{\bf  ($\Omega$ isomorphism)}
\label{theorem:omega-iso}
The open sets $\Omega^{(1)} \subseteq Z_1$ and $\Omega^{(2)} \subseteq Z_2$ are isomorphic.
\end{thm}
\begin{proof}
$\Omega^{(1)}$ and $\Omega^{(2)}$ both equal the quotient of $R^{(1)} \cap R^{(2)}$ by the action of $\mathbf{G}_m^\beta.$
\end{proof}

{\bf Batyrev-Nill's Example 5.1: torsor.}  In the example of Batyrev and Nill, the torsor  $\mathbf{G}_m^1 \times \Omega^{(1)}$ is  the set of points over $\Omega^{(1)}$ for which $(s,t)$ is a null row vector of the matrix $F^{(1)}$.  For instance, if we write the entries of $F^{(i)}$ as $f^{(i)}_{ab}$, the torsor as a subset of $T$ over $\Omega^{(1)}$ is given by the equations
$$
%f^{(1)}_{11} + f^{(1)}_{12}  \ = \ 
%f^{(1)}_{21} + f^{(1)}_{22} \ = \ 
(s f^{(1)}_{21} - t f^{(1)}_{11})  =  0 \text{  \quad  and \quad  } (s f^{(1)}_{22} - t f^{(1)}_{12})  =  0.
$$
Over $\Omega^{(2)}$ in $T$ we need $s$ and $t$  to give a null column vector for $F^{(2)}$. So the equations are
$$
%f^{(2)}_{11} + f^{(2)}_{21}  \ = \ 
%f^{(2)}_{12} + f^{(2)}_{22} \ = \ 
(s f^{(2)}_{12} - t f^{(2)}_{11}) = 0
\text{  \quad  and \quad  }
(s f^{(2)}_{22} - t f^{(2)}_{21})  =  0.
$$

\section{Application to multiple mirror nef-partitions.}
\label{section:mm-nef}

Our result applied to multiple mirror nef-partitions completely resolves the question posted by Batyrev-Nill \cite{Batyrev-Nill} on the birationality of the resulting families of Calabi-Yaus.  This completes the groundbreaking work of Li \cite{li-2016} by removing any need for the assumptions that the Calabi-Yaus and  a certain determinantal variety are irreducible.

\subsection{Nef-partition preparations.}
Below is a summary of the definitions and results we need from the theory of nef-partitions.  
More complete treatments can be found in Borisov \cite{Borisov},
Batyrev-Nill \cite{Batyrev-Nill},
Zhang \cite{li-2016} and
Clarke \cite{clarke-dual-fans}.

Nef-partitions are collections of compact, convex, integral polytopes.  
As their name suggests, they encode a collection of nef line bundles.
Their basic properties are that each polytope contains the origin and their sum is reflexive---that is to say the polar dual of their sum is also a compact, convex, integral polytope.
Borisov \cite{Borisov} introduced nef-partitions and at the same time a duality transformation on them.
This duality remains the primary means of obtaining examples of mirror Calabi-Yau manifolds.

The multiple mirror phenomenon arises when it is possible to translate the polytopes in a nef-partition in such a way that the result is again a nef-partition.  
Translation of polytopes does not alter the line bundles described by the polytopes; however, it does significantly alter the Borisov-dual nef-partition.

\subsubsection{Polytopes.}
The theory of nef-partitions is based on that of convex polytopes.  
Essentially all one needs are the notions of reflexivity and Minkowski sum, and a rich an complicated theory emerges.

{\bf Convex polytopes.}
Given a finite rank free abelian group $\mathfrak{A}$, an integral, compact, convex, polytope $P$ is the convex hull of a finite subset $S \subseteq \mathfrak{A}$ in the real vector space $V = \mathfrak{A} \otimes \mathbf{R}.$
We say $P$ is non-zero if $P \neq \{ 0 \}$.
The  polar dual polytope $P^\vee$ is the set of points $w$ in the dual space $V^\vee \cong \mathfrak{A}^\vee \otimes \mathbf{R}$ such that 
$$
\langle w, p \rangle \geq -1 \ \ \text{ for all $p \in P.$}
$$
An integral, compact, convex polytope $P$ is called reflexive if  $P^\vee$ is an integral, compact, convex polytope.

{\bf Nef-partitions.}
Let $N$ be a rank $n$ free abelian group.  A nef-partition is a collection $\nabla_1, \dots, \nabla_r \subseteq N_\mathbf{R} = N \otimes \mathbf{R}$ of non-zero, integral, compact, convex polytopes such that 
\begin{itemize}
\item $0 \in \nabla_k \text{ for all $k$},$ and 
\item the Minkowski sum $\sum_k \nabla_k$ is a reflexive polytope.
\end{itemize}
{Borisov-duality of nef partitions begins with $M = \operatorname{Hom}_\mathbf{Z}(N, \mathbf{Z})$,
and defines the dual nef partition $\Delta_1, \ldots, \Delta_r \subseteq \mathbf{M}_\mathbf{R}$ by 
\begin{equation}
\label{equation:nef-inequality}
\langle m, n \rangle + \delta(k',k) \geq 0 \text{ whenever $m \in \nabla_{k'}$ and $n \in \Delta_{k}$.}
\end{equation}
Here  $\delta(k',k)$ is the Kronecker delta.

\subsubsection{Geometry of a nef-partition.}
The polytopal data of a nef-partition defines geometric data of the kind that is considered in our theorem.  More accurately, one half of the geometry in the birationality statement.  
Specifically, we get a toric variety $Y$, a split bundle $\mathcal{V} = \mathcal{L}_1 \oplus \dotsm \oplus \mathcal{L}_r$, and a global section $g$ of this bundle.    
As before, the bundle and section are considered as a space $X$ and a function $W$ on it.  
Below, we will consider multiple mirror nef-partitions.  These will have the data of a pair of nef-partitions whose geometric data fits together as needed to apply the birationality theorem.

This geometrical data of a single nef-parition is presented first as a projective toric variety $Y_-$ with a split bundle $\mathcal{V}_- = (\mathcal{L}_-)_1 \oplus \dotsm \oplus (\mathcal{L}_-)_r$ over it.  Initially no section is specified.  However, the universal section $g$ of  $\mathcal{V}_-$ naturally lives in the pullback $\mathcal{V}$ of  $\mathcal{V}_-$ to an enlargement $Y = Y_- \times C$ of $Y_-$  by ``coefficients'' $C$.  At this point, the picture is as before: $X$ is then the total space of dual bundle $\mathcal{V}^\vee$, and $W$ is the section $g$ considered as a function on $X$.
Abstractly the affinization is as before, but we can give a more explicit description of it in terms of the cone $\kappa$ of global characters on $X_-$.

\begin{prop}{\bf (nef geometry)}
A nef partition $\Delta_1, \dotsm, \Delta_r \subseteq M_\mathbf{R}$ defines
\begin{itemize}
\item a projective toric variety $Y_-$ given by the polytope $\Delta_1+ \dotsm + \Delta_r$,
\item a completely split vector bundle $\mathcal{V}_- = (\mathcal{L}_-)_1 \oplus \dotsm \oplus (\mathcal{L}_-)_r$, and 
\item a choice of  trivialization by a global section ${e_k}$ for each $(\mathcal{L}_-)_k$ over the big torus $T_{Y_-}$ in $Y_-$ 
 \end{itemize}
 such that 
\begin{itemize}
\item the character group of \ $Y_-$ is $M$, 
\item a character $m$ times the trivializing section $e_k$ is a global section of $(\mathcal{L}_-)_k$ if and only if $m \in \Delta_k$, and
\item $e_k$ is a character in the natural toric structure on the total space $X_-$ of $\mathcal{V}_-^\vee$.
\end{itemize}
\end{prop}
\begin{proof}
Omitted
\end{proof}

{\bf Characters and affinization.}
The character group $\overline{M}$ on the total space $X_-$ of $\mathcal{V}_-^\vee$ defined by a nef-partition $\Delta_1, \dotsm, \Delta_r \subseteq M_\mathbf{R}$ is naturally identified with 
$$
M \oplus \bigoplus_k \mathbf{R} \cdot e_k
$$
for the trivializing global sections $e_k$.  

The integral points of  $\Delta_k + e_k \subseteq \overline{M}$ give the ``character basis'' of 
the global sections $\Gamma(Y_-, (\mathcal{L}_-)_k)$.  Consequently,  
the global characters on $X_-$ are the integral points of a cone $\kappa$ which is the $\mathbf{R}_{\geq 0}$-span of a polytope $\tilde{\Delta}$ which is the convex hull of the union the $\Delta_k + e_k$'s:
$$
\tilde{\Delta} = \operatorname{conv}( \cup_k   \Delta_k + e_k).
$$
The integral points of $\tilde{\Delta}$ itself give the character basis of 
the global sections $\Gamma(Y_-, (\mathcal{L}_-)_1 \oplus \dotsm \oplus (\mathcal{L}_-)_1).$
%Later the characters which appear in the expansion of the superpotential will be exactly these integral points.

Equipped with this description, the affinization $X_-^{(0)}$ of $X_-$ is exactly 
$$
X^{(0)}_- = \operatorname{Spec} \mathbf{C}[\kappa \cap \overline{M}].
$$

{\bf Coefficents.}
Since the integral points of $\tilde{\Delta}$  give the character basis for 
the global sections $\Gamma(Y_-, (\mathcal{L}_-)_1 \oplus \dotsm \oplus (\mathcal{L}_-)_k)$, 
these points give the natural index set for coefficients in the expansion of $W$.  With this in mind, 
we define the coefficient torus 
$$
C = \mathbf{C}[c_{\overline{m}}^{\pm 1}]_{\overline{m} \in \tilde{\Delta} \cap \overline{M}}.
$$
We then set $Y = Y_- \times C$,  $X = X_{-} \times C$, and 
$$
W = \sum_{\overline{m} \in  \tilde{\Delta} \cap \overline{M}} c_{\overline{m}}  \ \chi^{\overline{m}} \  \  \colon  \  \ X^{(0)} \to \mathbf{A}^1.
$$
Consequently,  $X^{(0)} =  X_{-}^{(0)} \times C$.

{\bf Batyrev-Nill's Example 5.1: indexing note.}  For this example, we use a slightly different index set.  
Essentially, we use the integral points of the convex hull $\operatorname{conv}( \cup_k   \Delta_k) \subseteq M_\mathbf{R}.$  This is done  so that our indices have two entries rather than four. The trade-off is that we are required to index the ``origin terms'' differently.  Below the $\Delta$-sets are given explicitly, and one can match  integral points there with the indices of the coefficients in the expansion of $W$ at the end of Subsection \ref{subsection:geometries}.

\subsubsection{Multiple mirror geometry.}
Our constructions above have been made in preparation for considering two nef partitions 
$\Delta^{(1)}_1,   \dotsc ,  \Delta^{(1)}_r$ and
$\Delta^{(2)}_1,   \dotsc ,  \Delta^{(2)}_r$
in $M_\mathbf{R}$.  For multiple mirror nef-partitions this pair is not arbitrary. 
Their relationship is that their Borisov dual nef-partitions are translates of each other. 
To describe this precisely, we begin with a single nef-partition
$\nabla^{(1)}_1, \dotsc, \nabla^{(1)}_r \subseteq N_{\mathbf{R}}$.
Recall, $N = M^\vee = \operatorname{Hom}_\mathbf{Z}(M, N)$.

Given a nef-partition $\nabla^{(1)}_1, \dotsc, \nabla^{(1)}_r \subseteq N_{\mathbf{R}}$,
it is sometimes possible to choose elements $n_k \in N$, not all of which are zero, such that the translations $$
\nabla^{(2)}_k = \nabla^{(1)}_k + n_k
$$ 
again form a nef-partition.  
A consequence of the definition of a nef-partition is that we can characterize such  $n_k$'s as those which satisfy two properties:
\begin{enumerate}
\item $-n_k \in \nabla^{(1)}_k$ for all $k$, and 
\item $\sum_k n_k = 0$.
\end{enumerate}

Geometrically, the existence of such $n_k$'s creates a multiple mirror situation.  With regard to the $\nabla$'s, translation by the $n_k$'s has little effect.  The Minkowski sum is preserved:
$$
\nabla^{(1)}_1 +  \dotsm + \nabla^{(1)}_r = 
\nabla^{(2)}_1 +  \dotsm + \nabla^{(2)}_r.
$$
So the toric variety $Y'$ is unchanged (we are writing the prime $'$ here to remind us that we are beginning with a nef-partition in $N_\mathbf{R}$ rather than $M_{\mathbf{R}}$).  
Furthermore, the bundles over $Y'$ are the same and so the total space $X'$ is the same.  
The only difference is in the choice of trivializing sections $(e')^{(i)}_k$. These are 
characters in  $\overline{N}'$ on $X'$, and are related by 
$$
(e')^{(1)}_k = n_k + (e')^{(2)}_k.
$$
The chief significance of this is that the families of Calabi-Yaus $(Z')^{(i)}$ are exactly the same for both $\nabla^{(i)}$'s.  
However, the mirror families $Z^{(i)}$ defined by the Borisov duals $\Delta^{(i)}$ are different.
This is why nef-partitions $\Delta^{(i)}$ whose Borisov duals are related by this kind of translation are called multiple mirror nef-partitions.

{\bf Multiple mirrors.}
Two nef partitions $\Delta^{(1)}_1, \dotsc, \Delta^{(1)}_k$ and $\Delta^{(2)}_1, \dotsc, \Delta^{(2)}_k$ in $M_\mathbf{R}$
are  multiple mirrors if their Borisov duals
$\nabla^{(1)}_1, \dotsc, \nabla^{(1)}_k$ and $\nabla^{(2)}_1, \dotsc, \nabla^{(2)}_k$ in $N_\mathbf{R}$
are related by 
$$
\nabla^{(2)}_k = \nabla^{(1)}_k + n_k
$$ 
for $n_k\in N$ such that 
$$
\sum_k n_k = 0.
$$

\begin{prop}{\bf (multiple mirror geometry)}
If $\Delta^{(1)}_1, \dotsc, \Delta^{(1)}_k$ and $\Delta^{(2)}_1, \dotsc, \Delta^{(2)}_k \subseteq M_\mathbf{R}$
are multiple mirror nef-partitions, then the character groups $\overline{M}$,  the cones of global characters $\kappa$, and the global section characters $\tilde{\Delta}$ are the same.
Thus the space $X_0$ and the function \ $W$ are the same for both nef-partitions, and we are in the situation of 
the ``Geometry'' Subsection \ref{subsection:geometries}.  Furthermore, Assumption 1 
of Subsection \ref{subsection:assumptions} is satisfied, and for any specialization of the coefficient variables $c_{\overline{m}}$ Assumption 2 is satisfied.
\end{prop}
\begin{proof}
Omitted.
\end{proof}

{\bf Batyrev-Nill's Example 5.1: nef-partitions.}  The picture here, redrawn from Batyrev-Nill \cite{Batyrev-Nill}, illustrates the multiple mirror nef-partitions of their Example 5.1:
\begin{center}
\begin{tabular}{ccc}
&  & \\
\includegraphics{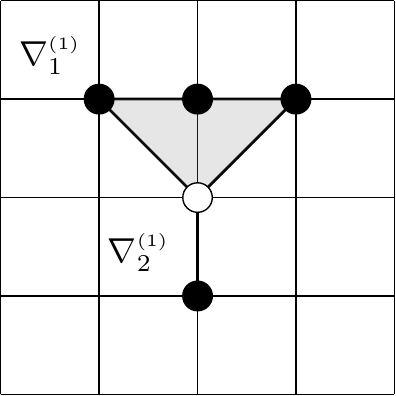} & \qquad & 
\includegraphics{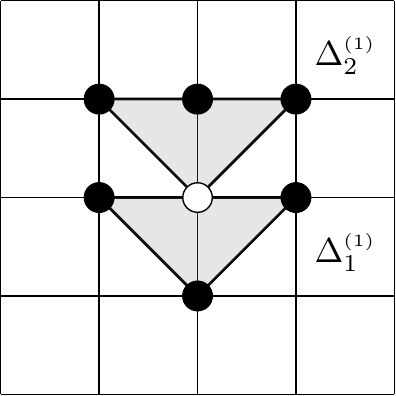} \\
& & \\
\includegraphics{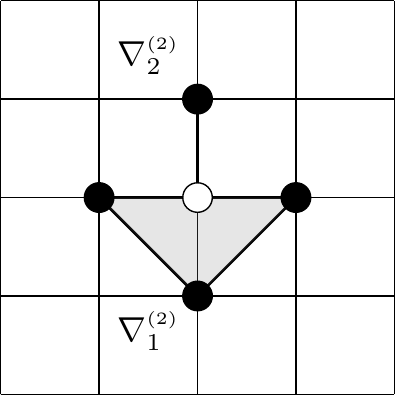} & \qquad & 
\includegraphics{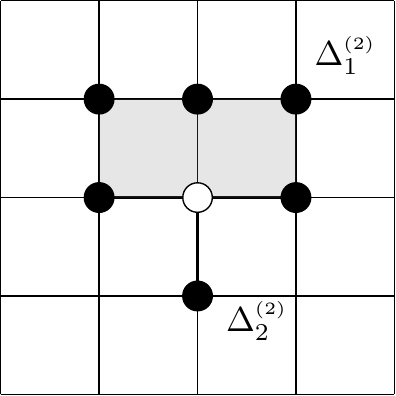} \ .\\
\end{tabular}
\end{center}
Here the vectors translating the $\nabla$'s are $n_1 = (0,-1)$ and $n_2 = (0,1).$
Notice that the integral points in the convex hull of the $\Delta$ nef-partitions were used as the index set for the coefficients in $W$.  
Using these points, rather than those of $\tilde{\Delta} \cap \overline{M}$, explains  why the ``origin'' coefficients are indexed differently.  We made this choice so that our indexing vectors would have two entries rather than four.

\subsection{$D$-graphs of multiple mirror nef-partitions.}
\label{subsection:d-graph}

The matrices $(W)$ obtained from multiple mirror nef-partitions have a specific form. 
Later, this will allow us to conclude the non-emptiness of the sets $\Omega$. 
What we need to know about $(W)$ can be read off of the graph $D$.  
For multiple mirror nef-partitions, each of the connected components $D_{(j)}$ is strongly connected.
This structure is a consequence of the fact the elements $n_k$ giving the multiple mirror data can be used to detect arrows in $D$.  
We also use a trick we call ``coarsening''  in our argument.  Later, the strong connectivity of the components of $D$ will allow us to apply the Perron-Frobenius theorem to produce a point in the set $O^{(1)}$. This guarantees the non-emptiness of the   $\Omega$-sets.

\begin{prop}{\bf ($W_{ab}$'s and $n_k$'s)} 
The characters appearing in the terms of $W_{kk}$ correspond to the integral points of $\Delta^{(1)}_k \cap \Delta^{(2)}_k$, and when $a \neq b$ the characters appearing in in terms $W_{ab}$ correspond to the non-zero integral points of $\Delta^{(1)}_a \cap \Delta^{(2)}_b$.  Furthermore when $a \neq b$, a non-zero integral point $m$ lies in the intersection $\Delta^{(1)}_a \cap \Delta^{(2)}_b$ if and only if 
$$-\langle m, n_b \rangle = \langle m, n_a \rangle = 1.$$
\end{prop}
\begin{proof}
The first statements  can be found  in the references above on nef-partitions, and follow from an elementary analysis of the global section polytope $\tilde{\Delta} \subseteq \kappa.$

For the claims about using the  $n_k$'s to detecting non-zero elements in
the sets $\Delta^{(1)}_a \cap \Delta^{(2)}_b$, begin with the fact that every  non-zero, integral element of a reflexive polytope is an element of the boundary of the polytope.  
So for every non-zero, integral element $n$ in a reflexive polytope, there is a non-zero integral element $m$ in the polar dual such that 
$$
\langle m, n \rangle = -1.
$$
In the case of nef-partitions, the polar dual of the sum $\sum_k \nabla^{(i)}_k$ is the convex hull of the union $\cup_k \Delta^{(i)}_k$, and any integral point in this convex hull is a member of some $\Delta^{(i)}_k$.  In this multiple mirror situation, the sum and convex hull polytopes are independent of $i$.

The membership $-n_a \in \nabla^{(1)}_a$ guarantees that if $n_a \neq 0$, then there is an integral element $m \in \cup_k \Delta^{(i)}_k$ such that $\langle m, n_a \rangle = 1.
$
Furthermore, the inequalities of Borisov duality (labelled (\ref{equation:nef-inequality}) above) guarantee that $m \in \Delta^{(1)}_a$ and for any $b \neq a$ we have 
$\langle m, -n_b \rangle \geq 0.$  The vanishing $\sum_k n_k = 0$ means that there is exactly one $b$ for which 
$\langle m, -n_b \rangle \geq 0,$ and all the rest are zero.  Finally, reversing the rolls of the $\Delta^{(1)}$'s and  $\Delta^{(2)}$'s 
means replacing the $n$'s with their negatives; so we can conclude that $m \in \Delta^{(1)}_a \cap \Delta^{(2)}_b$.
\end{proof}

\begin{prop} {\bf (coarsening the partition)}
\label{prop:coarsening}
If $\Delta_1, \dotsc, \Delta_r$ is a nef-partition and we have an equivalence relation on $\{ 1, \dotsc, r \}$ with equivalence classes $[a]$,
then setting
$$
\Delta_{[a]} = \sum_{a' \in [a]} \Delta_{a'}
$$
is again a nef partition indexed by the equivalence classes $[a]$.  
Furthermore, if $n_1, \dotsc, n_r$ is multiple mirror data for the $\Delta_a$'s, then 
$$
n_{[a]} = \sum_{a' \in [a]} n_{a'}
$$
is multiple mirror data for the $\Delta_{[a]}$'s.  Finally,
the $D$ graph corresponding to the $\Delta_{[a]}$'s and $n_{[a]}$'s is the quotient of the graph corresponding to the 
the $\Delta_a$'s and $n_a$'s.
\end{prop}
\begin{proof}
Omitted.
\end{proof}

{\bf Strongly connected graphs.}  A directed graph is called strongly connected if there is a path in each direction between any two vertices.

\begin{prop}{\bf (strong connectedness)}
The connected components of the graphs $D$ which arise from multiple mirror nef-partitions are strongly connected, and each node has a loop.
\end{prop}
\begin{proof}
We know the nodes of $D$ are looped since $\chi^{e_k}$ appears in the expansion of $W_{kk}$.
Otherwise there is an arrow from the node $a$ to a different node if and only if $\Delta^{(1)}_a \cap \Delta^{(2)}_b$ is non-zero for some $b$.  This is equivalent to saying that there is some element $m \in \Delta^{(1)}_a$ such that
$$
 \langle m, n_a \rangle = 1.
$$
This always happens because $-n_a$ is inward normal to  a face of $\Delta^{(1)}_a$.
 So every node has a non-loop arrow exiting it.  
 
In other words, every row of $(W_{(j)})$ has at least two non-zero entries: the diagonal and an ``exiting'' arrow.  
Reversing the rolls of the $\Delta^{(1)}$'s and the  $\Delta^{(2)}$'s leads to the same statement for the columns, i.e. every node has a non-loop arrow entering it.

Now assume $D_{(j)}$ is not strongly connected.  Consider the ``condensation'' graph: this is the quotient graph $D/\sim$ formed by identifying vertices in the same strongly connected component.
To be clear,   $D/\sim$ has nodes equal to the equivalence classes
$$
[k] = \{ k' | k' \text{ is in the same strongly connected component as } k \}, 
$$
 and there is an edge from $[a]$ to $[b]$ if and only if there  are elements $a' \in [a]$ and $b' \in [b]$ with and edge from $a'$ to $b'$ in $D$.  The condensation graph is acyclic, apart from the loops, and it is the graph that arises from the multiple mirror nef-partitions given by the coarsening:
\begin{itemize}
\item the nef-partition $\nabla^{(i)}_{[k]} = \sum_{k' \in [k]} \nabla^{(i)}_{k'}$, and
\item the translations $n_{[a]} = \sum_{k' \in [k]} n_{k'}$.
\end{itemize}
However, the only way the graph of nef-partition may be acyclic in this way is if its connected components
each have a single vertex, since otherwise it would have a node with no non-loop exiting arrow.  This means
$D_{(j)}/\sim$ is a single vertex, so $D_{(j)}$ is strongly connected and the proof is complete.
\end{proof}

\subsection{Multiple mirror theorem.}

Our application to the birationality of Calabi-Yau families arising from multiple mirror nef-partitions proceeds in two steps.
We first establish that the $\Omega$ sets are non-empty by explicitly producing a point in $O^{(1)}.$ 
This point exists by virtue of the fact that the connected components of  the $D$-graph are strongly connected, so we can make an argument via the Perron-Frobenius theorem.
The second step is essentially the observation that the vanishing locus $Z^{(i)}$ in $Y^{(i)}$ of the universal section $g$ is irreducible.

{\bf Perron-Frobenius theorem.}  
The Perron-Frobenius theorem applies to square matrices $A$ with two properties:  
\begin{itemize}
\item the entries of $A$ are non-negative, and
\item the graph formed by drawing an arrow $k \to k'$ if and only if $a_{kk'} \neq 0$ is strongly connected. 
\end{itemize}
Recall, a directed graph is called strongly connected if there is a path in each direction between any two vertices.

Given these non-negativity and strong connectivity conditions on a matrix $A$, the consequence of Perron-Frobenius theorem that is relevant to us is 
\begin{itemize}
\item $A$ has a simple real eigenvalue $r > 0$ such that the column eigenvector $v$ and row eigenvector $h$ both have all positive entries. 
\end{itemize}
More details can be found in the original works of Perron and Frobenius \cite{perron, frobenius}, and a modern account is in Meyer \cite{meyer}.

\begin{prop}{\bf (nef $\Omega$ non-emptiness)}
\label{proposition:omega-non-empty}
In the case of multiple mirror nef-partitions, the $\Omega$ sets are non-empty.
\end{prop} 
\begin{proof}
The torus $T$ is the product of a torus $T_-$ whose character group is $\overline{M}$ and the coefficient torus $C$.
We will show that we can choose a point $t \in T$ that lies in $O^{(1)}$.  
This will guarantee $\Omega^{(1)}$, and thus $\Omega^{(2)}$ are non-empty.

For a point  $t \in T$ to lie in $O^{(1)}$ we must have
$$
(W) \begin{pmatrix} 1 \\ \vdots \\ 1 \end{pmatrix} = 0,
$$
$$
 h(W) = 0 \text{ for $h$ with all entries $\neq 0$,}
$$
and the rank of $(W)$ must be $k-\beta$.

To begin, for each non-zero entry $W_{ab}$ of $(W)$ substitute $c_\chi = w_{ab}/ \#(\Lambda^{(1)}_a \cap \Lambda^{(2)}_b)$ for an  indeterminate 
$w_{ab}$.  
Now specializing the characters to $1 \in  T_{-}$ renders $(W)$ so that variables $w_{ab}$ are sitting
 in its non-zero entries.  
We now we need only find values for the  $w_{ab}$'s that give the matrix the needed properties.  
At first we can set $w_{ab} =1$ for $a \neq b$.
Then $(W)$ has the form 
$$
(W) = \operatorname{diag}(w_{kk}) + A.
$$
The graph obtained from $A$ equals the graph $D$ with the loops removed.
In particular, $A$ has a block diagonal form where the blocks have strongly connected graphs.
This means that we can assign the value $-r_j$ the $w_{kk}$'s in the $j^\text{th}$ block, where $r_j > 0$ is the eigenvalue from the Perron-Frobenius theorem for the $j^\text{th}$ block of $A$.
All together, this means 
$$
(W) = \operatorname{diag}(-r_j) + A
$$
has rank $k-\beta$ and annihilates a column vector $v$ and a row vector $h$ all of whose entries are positive---in particular, non-zero.

So far, we have specialized to some $(1,c) \in T_- \times C$. Now we adjust this point  so that it lies in $O^{(1)}$.  To do this, use the second factorization
$$
(W) =  F^{(2)} \operatorname{diag}(x)
$$
and scale the  $x$'s so that they equal to the entries of $v$.  This way  
$$
(W) \begin{pmatrix} 1 \\ \vdots \\ 1 \end{pmatrix} = 0,
$$
$$
 h(W) = 0 \text{ for $h$ with all entries $\neq 0$,}
$$
and the rank of $(W)$ is $k-\beta$.  
We now lie in $O^{(1)}$, as needed.
\end{proof}

\begin{thm}{\bf (birationality of nef multiple mirror families)}
\label{theorem:birat-nef}
In the situation of multiple mirror nef-partitions, the vanishing locus $Z^{(i)}$ of the universal section $g^{(i)}$ is irreducible.  So the set $\Omega^{(i)}$ is dense.
Consequently, there is an open dense subset of the space of coefficients for which $(Z^{(1)})_c$ and $(Z^{(2)})_c$ are birational when $c$ is in this set.
\end{thm}
\begin{proof}
The main observation here is that the universal complete intersection $Z^{(i)}$ in $Y_-^{(i)} \times C$ is irreducible.  
This can be seen by considering the fiber $(Z^{(i)})_{y^{-}_i}$ over a point $y^{-}_i \in Y_-^{(i)}$.  
This defines a linear subspace of sections which vanish at $y^{-}_i$.  
The linear subspace picks out a subtorus of $C$.  This family of subtori is constant rank over $T_{Y^{(i)}}$.
So 
$Z^{(i)}$ is an irreducible torus bundle over $Y_-^{(i)}$.

Now we know that  the $\Omega$ sets are open and dense.  
So we can find an open dense subset of $C$  such that  $(Z^{(i)})_c \cap \Omega$ is dense when $c$ is in this set.  
With $c$ specialized, both Assumption 1 and Assumption 2 are satisfied, so 
$(Z^{(1)})_c$ and $(Z^{(2)})_c$ are birational.
\end{proof}

{\bf Batyrev-Nill's Example 5.1: $O_1 \neq \varnothing$ and the torsor ``direction.''} In the example of Batyrev and Nill, the specializations we use in these proofs are 
$$
A = 
\begin{pmatrix}
0 & 1\\
1 & 0
\end{pmatrix}
$$
and 
$$
(W) = 
\begin{pmatrix}
w_{11} & 0\\
0 & w_{22}
\end{pmatrix}+
\begin{pmatrix}
0 & 1\\
1 & 0
\end{pmatrix}.
$$
Setting $w_{11} = w_{22} = -1$ gives a point in $R^{(1)} \cap R^{(2)}$.  
Here $s=t$ gives the $\mathbf{G}_m^1$ fiber.  Recall
at other points, if the null row vector of $F^{(1)}$ is $(h_1, h_2)$, then $h_2s=h_1t$ gives the ``direction'' of the $\mathbf{G}_m^1$ fiber.

\section{Remark on multiple mirror Fanos.}
In a multiple mirror configuration for which the number $\beta$ of connected components of $D$ is greater than one,  we can compare the Fano complete intersections defined by choosing a non-empty proper subset  $B \subset \{ 1, \dotsc, \beta \}$
and considering the vanishing loci of the universal sections of 
$$
\bigoplus_{k \in \bigcup_{j \in B} D_{(j)} } \mathcal{L}^{(i)}_k.
$$
In this context the natural function to consider is the sum of block functions
$$
W_(B) = \sum_{j \in B} W_{(j)}.
$$
Our results are sufficient to show that for fixed $B$, the Fanos are birational.

 This  also constitutes a ``multiple mirror'' situation, since these Fanos arise as mirrors to the same Landau-Ginzburg model.
 On one hand, Kontsevich's homological mirror symmetry \cite{kontsevich-hms} combined with Bondal-Orlov's reconstruction theorem \cite{bondal-orlov-reconstruction} suggests that  they are more than just birational---they are isomorphic.  However, the details of homological mirror symmetry in this situation are not yet sufficiently developed to draw any definite conclusions.

\newpage
\footnotesize
\bibliography{BB-birat}
\bibliographystyle{halpha}  

\noindent
{Department of Mathematics, Drexel University, Philadelphia, PA 19104\\
\texttt{pclarke@math.drexel.edu}}

\end{document}